\documentclass{article}
\usepackage{amsmath,amssymb,amsthm,pxfonts,color}
\newtheorem{defn}{Definition}
\newtheorem{thm}{Theorem}

\usepackage{ulem}

\title{
Local asymptotic normality property for fractional Gaussian noise under high-frequency observations
}
\author{Alexandre Brouste$^\ast$ and Masaaki Fukasawa$^\dagger$ \\
{\small
$^\ast$ D\'epartement de Math\'ematiques, Universit\'e du Maine }\\
 {\small $^\dagger$ Graduate School of Engineering Science, Osaka University}
 \\
 {\small 1-3 Machikaneyama, Toyonaka, Osaka, 560-8531 JAPAN}\\
 {\small  fukasawa@sigmath.es.osaka-u.ac.jp}
}
\begin{document}
\maketitle

\begin{abstract}
 Local Asymptotic Normality (LAN) property for fractional Gaussian noise under high-frequency  observations is proved with a non-diagonal rate matrix  depending on the parameter to be estimated.
In contrast to the LAN families in the literature,  non-diagonal rate matrices are inevitable.
\end{abstract}

\section{Introduction}
The theory of Local Asymptotic Normality (LAN) provides a powerful framework under which we are able to discuss asymptotic optimality of estimators.
When the LAN property holds true for a statistical experiment, minimax theorems \cite{Hajek,LeCam} can be applied and a lower bound for the variance of the estimators can be derived. Beyond the classical IID setting \cite{IH}, the LAN property (or Local Asymptotic Mixed Normality property) has been proved for various statistical models including  linear processes \cite{KT},  ergodic Markov chains \cite{Roussas}, ergodic diffusions \cite{Gobet1,Kutoyants}, diffusions under high-frequency observations \cite{Gobet2} and several L\'evy process models \cite{CG, Kawai2, KM}.
The LAN property for fractional Gaussian noise (fGn) was obtained in \cite{Cohen} under the large sample observation scheme.  In this setting, Maximum Likelihood (ML) and Whittle sequences of estimators achieve optimality \cite{Dahlhaus, Dahlhaus2, Lieberman}.

The statistical experiment of observing fGn under high-frequency scheme has not been well understood in the literature, despite that high-frequency data has attracted much attention recently due to their increasing availability. At high-frequency, scaling effects from the variance and from  the self-similarity of the fGn are melting.  The singularity of the joint estimation of diffusion coefficient and Hurst parameter was already noticed in~\cite{Berzin,Istas}. 
A weak LAN property with a singular Fisher matrix was obtained in \cite{Kawai}.
Due to this singularity, no minimax theorem can be applied and in particular,
it has been unclear whether the ML estimator enjoys any kind of optimality property.

In this paper, we prove that the  statistical experiment in fact enjoys the LAN property.
To discuss the difference from \cite{Kawai},
let us be more precise in the definition of the LAN property.

 \begin{defn}\label{def:LAN}
  Let $\Theta \subset \mathbb{R}^d$.
 A family of measures $\{P^n_\theta, \theta \in \Theta\}$
 is called locally asymptotically normal (LAN) at $\theta_0 \in \Theta$ if there exist  {\bf nondegenerate} $d\times d$ matrices $\varphi_n(\theta_0)$ and $I(\theta_0)$ such that for any $u \in \mathbb{R}^d$,
 the likelihood ratio
\begin{equation*}
Z_n(u) = \frac{\mathrm{d}P^n_{\theta_0 + \varphi_n(\theta_0)u}}{\mathrm{d}P^n_{\theta_0}}
\end{equation*}
 admits the representation
 \begin{equation}\label{eq:1}
  Z_n(u) = \exp\left(
		\langle u , \zeta_n(\theta_0)\rangle - \frac{1}{2}\langle I(\theta_0)u,u\rangle + r_n(\theta_0,u)\right),
 \end{equation}
 where
 \begin{equation}\label{eq:2}
  \zeta_n(\theta_0) \to \mathcal{N}(0, I(\theta_0)), \ \   r_n(\theta_0,u) \to 0
 \end{equation}
 in law under $P^n_{\theta_0}$.
 %The family is called LAN at  $\theta_0 \in \Theta$ if it is LAN at $\theta_0 \in \Theta$ for some {\bf nondegenerate} matrices $\varphi_n(\theta_0)$ and $I(\theta_0)$.
 \end{defn}
 This definition of the LAN property
 is equivalent to the one given in \cite{IH}.
 The matrix $\varphi_n(\theta_0)$ is often called the rate matrix.
Remark that the  nondegeneracy of $\varphi_n(\theta_0)$ and $I(\theta_0)$ is essential in the following minimax theorem  due to Hajek~\cite{Hajek} and Le Cam~\cite{LeCam}, which implies in particular the asymptotic efficiency of the ML sequence of estimators in regular models.

 \begin{thm}\label{thm:minimax}
  Let the family of measures
  $\{P^n_\theta, \theta \in \Theta\}$, $\Theta \subset \mathbb{R}^d$,
  be LAN at $\theta_0 \in \Theta$ for nondegenerate matrices $\varphi_n(\theta_0)$ and $I(\theta_0)$. Then, for any sequence of estimators $\hat{\theta}_n$,
  \begin{equation*}
   \liminf_{\delta \to 0}
    \liminf_{n\to \infty}
    \sup_{|\theta- \theta_0| < \delta}
    E_\theta\left[
l\left( \varphi_n(\theta_0)^{-1}(\hat{\theta}_n - \theta)\right)
			\right] \geq
    \int_{\mathbb{R}^d}l\left(I(\theta_0)^{-1/2}z\right)\phi(z)\mathrm{d}z
  \end{equation*}
  for any symmetric, nonnegative quasi-convex function $l$ with
  $\lim_{|z|\to \infty} e^{-\epsilon |z|^2}l(z) = 0$ for all $\epsilon > 0$,
  where $\phi$ is the density of the $d$-dimensional standard normal distribution.
 \end{thm}

 For the fGn under high-frequency observations, we consider the estimation of the parameters $(H,\sigma)  \in \Theta = (0,1) \times (0,\infty)$, where $H$ is the 
 Hurst parameter and $\sigma$ is the diffusion coefficient.
 It was shown in  \cite{Kawai} that the family  of measures $\{P^n_{(H,\sigma)},(H,\sigma)  \in \Theta\}$
 satisfies both conditions \eqref{eq:1} and \eqref{eq:2} at any $(H,\sigma) \in \Theta$ for
 \begin{equation*}
  \varphi_n(H,\sigma) =
     \begin{pmatrix}
      \frac{1}{\sqrt{n}\log \Delta_n} & 0 \\ 0 & \frac{1}{\sqrt{n}}
       \end{pmatrix} \quad \mbox{and} \quad
		    I(H,\sigma) = \begin{pmatrix}
2 & 2/\sigma
				 \\ 2/\sigma & 2/\sigma^2
				\end{pmatrix},
 \end{equation*}
where $n$ is the sample size and
 $\Delta_n$ is the length of sampling intervals.
 Note that $I(H,\sigma)$ is singular. On the one hand, this result does not imply that the family is LAN in the sense of Definition~\ref{def:LAN}. On the second hand, Theorem~\ref{thm:minimax} can not be applied with this result.

 To the best of our knowledge, for any LAN family in the literature, it is always possible to take
 a diagonal rate matrix.
 A typical example is the LAN property for the IID setting where 
 $$\varphi_n(\theta_0) = \frac{1}{\sqrt{n}} \mathbf{I}_d.$$ 
 Here  $\mathbf{I}_d$ is the $d\times d$ identity matrix.
 %\begin{equation*}
%			 \varphi_n(\theta_0) = \begin{pmatrix}
%
%						\frac{1}{\sqrt{n}} & 0 & \ldots & 0 \\
%						0 & 			\frac{1}{\sqrt{n}} & \ddots & \vdots \\
%\vdots & \ddots & \ddots & 0 \\						
%0 & \ldots & 0 & 	\frac{1}{\sqrt{n}}				       \end{pmatrix}
 %\end{equation*}
However,  both in Definition~\ref{def:LAN} and Theorem~\ref{thm:minimax}, $\varphi_n(\theta_0)$ is required only to be nondegenerate.
 In this paper, we prove the LAN property for the statistical experiment of observing fGn under high-frequency scheme for a certain class of
 {\bf non-diagonal} matrices $\varphi_n(\theta_0)$ depending on $\theta_0$.
 In particular, Theorem~\ref{thm:minimax} can be applied.
Non-diagonal rate matrices are inevitable because the Fisher matrix is singular.
 
 Basics for the fractional Brownian motion and the fGn are recalled in Section 2. The statistical experiment under high-frequency scheme is described and the LAN property result is proved in Section~3.  Efficient rates for the estimation of $H$ and $\sigma$ are given in Section 4 giving the optimality of the ML sequence of estimators. 
 
\section{Fractional Brownian motion and fractional Gaussian noise}

Here we briefly review the basics of the fractional Brownian motion, fractional Gaussian noise and their large sample theory. A centered Gaussian process $B^H$ is called a fractional Brownian motion with Hurst parameter $H$ if
\begin{equation*}
 E[B^H_t B^H_s] = \frac{1}{2}(|t|^{2H} + |s|^{2H}- |t-s|^{2H})
\end{equation*}
for all $t, s \in \mathbb{R}$.
 Such a process exists and is continuous for any $H \in (0,1)$ by Kolmogorov's extension and continuity theorems.
 The process enjoys a self-similarity property : for any $\Delta > 0$ and $t \in \mathbb{R}$,
 \begin{equation*}
 B^H_{t + \Delta} - B^H_t \sim \Delta^H B^H_1
 \end{equation*}
 in law. The spectral density of regular increments, that is, the function $f_H$ characterized by
 \begin{eqnarray*}
  E[B^H_{1}(B^H_{k+1} - B^H_{k})] &=&  \frac12 \left( |k+1|^{2H}-2 |k|^{2H} + |k-1|^{2H}\right) \\
   &=& \frac{1}{2\pi}\int_{-\pi}^{\pi} e^{\sqrt{-1}k\lambda}f_H(\lambda)\mathrm{d}\lambda, \ \ k \in \mathbb{Z}
 \end{eqnarray*}
 is given by
 \begin{equation*}
  f_H(\lambda) =  C_H 2 (1-\cos(\lambda)) \sum_{k \in \mathbb{Z}}  | \lambda + 2k \pi |^{-2H-1}  \quad  \mbox{with} \quad C_H= \frac{\Gamma(2H+1) \sin (\pi H)}{2\pi} .
 \end{equation*}
 %A simple large sample theory of the fractional Brownian motion
 %is concerned with estimating $(H,\sigma)$ from data
 %\begin{equation*}
  %\sigma B^H_{\Delta},   \sigma B^H_{2\Delta}, \dots,
   %  \sigma B^H_{n\Delta}
 %\end{equation*}
 For a fixed interval $\Delta > 0$, assume we observe the increments
 \begin{equation*}
X_n =
  \left(
   \sigma B^H_{\Delta}, 
   \sigma B^H_{2\Delta} - \sigma B^H_{\Delta},
   \sigma B^H_{3\Delta} - \sigma B^H_{2\Delta}, \dots,
      \sigma B^H_{n\Delta} - \sigma B^H_{(n-1)\Delta}
  \right),
 \end{equation*}
 where $(H,\sigma)$ is unknown. The random vector $X_n$ is called fractional Gaussian noise (fGn).
 Denote by  $T_n(H)$ the covariance matrix of
\begin{equation*}
 \frac{X_n}{\sigma \Delta^H},
\end{equation*}
of which the distribution does not depend on $\sigma$ and $\Delta$  by the self-similarity property.
The $(i,j)$ element of $T_n(H)$ is given by
\begin{equation*}
 \frac{1}{2\pi}\int_{-\pi}^{\pi} e^{\sqrt{-1}(i-j)\lambda}f_H(\lambda)\mathrm{d}\lambda.
\end{equation*}
Let us suppose $\Delta = 1$ for brevity. The log-likelihood ratio is then given by
\begin{equation*}
 \begin{split}
 \ell_n(H,\sigma) = &
  -\frac{n}{2}\log(2\pi) - n\log \sigma \\
  & - \frac{1}{2}\log | T_n(H) | - \frac{1}{2\sigma^2}\langle
  X_n,T_n(H)^{-1} X_n \rangle.
 \end{split}
\end{equation*}
Let
\begin{equation*}
 \begin{split}
 &A_n = \sqrt{n}
  \left\{
\frac{1}{n\sigma^2}\langle
  X_n,T_n(H)^{-1} X_n \rangle -1
		\right\}, \\
  &  B_n =  \frac{1}{\sqrt{n}} \left\{
  \frac{1}{2} \partial_H
  \log | T_n(H) | +
 \frac{1}{2\ \sigma^2}\langle
  X_n,\partial_H\{T_n(H)^{-1}\} X_n \rangle\right\}.
 \end{split}
\end{equation*}
The score function is then given by
\begin{equation*}
 \nabla  \ell_n = 
  \begin{pmatrix}
   \partial_H  \ell_n \\
   \partial_\sigma  \ell_n 
  \end{pmatrix}
     = \begin{pmatrix}
	 - B_n \sqrt{n}  \\
     A_n \sqrt{n} \sigma^{-1}
       \end{pmatrix}.
\end{equation*}
Let $P^n_{(H,\sigma)}$ be the measure on $\mathbb{R}^n$ induced by $X_n$.

\begin{thm}
 The family of measures $\left\{P^n_{(H,\sigma)}; (H,\sigma) \in (0,1)\times (0,\infty)\right\}$ is LAN at any $(H,\sigma) \in (0,1)\times (0,\infty)$ with
  \begin{equation*}
  		   \varphi_n(H,\sigma) = \begin{pmatrix}
\frac{1}{\sqrt{n}} & 0
							\\
		0 & \frac{1}{\sqrt{n}}	     
		 \end{pmatrix}
		 \end{equation*}
		 and
\begin{equation*}I(H,\sigma) =
		 \begin{pmatrix}
		  \frac{1}{4\pi} \int_{-\pi}^{\pi} |\partial_H \log f_H(\lambda)|^2 \mathrm{d}\lambda &
		  \frac{1}{2\pi\sigma} \int_{-\pi}^{\pi} \partial_H\log f_H(\lambda)\mathrm{d}\lambda \\
		  \frac{1}{2\pi\sigma} \int_{-\pi}^{\pi} \partial_H\log f_H(\lambda)\mathrm{d}\lambda &
		 \frac{2}{\sigma^2}
		 \end{pmatrix}.   
  \end{equation*}
 In particular, $(A_n, B_n)$ converges in law to a centered normal distribution with covariance
 \begin{equation} \label{AB}
J(H) =   		 \begin{pmatrix}
		 2 &
		-  \frac{1}{2\pi} \int_{-\pi}^{\pi} \partial_H\log f_H(\lambda)\mathrm{d}\lambda \\
		-  \frac{1}{2\pi} \int_{-\pi}^{\pi} \partial_H\log f_H(\lambda)\mathrm{d}\lambda &
		  \frac{1}{4\pi} \int_{-\pi}^{\pi} |\partial_H \log f_H(\lambda)|^2 \mathrm{d}\lambda
		 \end{pmatrix}.   
 \end{equation}
\end{thm}

\begin{proof}
See \cite{Cohen}.
\end{proof}

\section{The LAN property in high-frequency observation}
%Now, suppose we observe $\sigma B^H$
%at time $0, \Delta_n, 2\Delta_n, \dots, n\Delta_n$.
Let $X_n$ be again the observed fractional Gaussian noise
\begin{equation*}
 X_n =
  \left(
   \sigma B^H_{\Delta_n}, 
   \sigma B^H_{2\Delta_n} - \sigma B^H_{\Delta_n},
   \sigma B^H_{3\Delta_n} - \sigma B^H_{2\Delta_n}, \dots,
      \sigma B^H_{n\Delta_n} - \sigma B^H_{(n-1)\Delta_n}
  \right).
\end{equation*}
Here we consider high-frequency asymptotics, which means $\Delta_n \to 0$ as $n\to \infty$.
The parameters  to be estimated are still $H \in (0,1)$ and $\sigma > 0$.
As before, the distribution of
\begin{equation*}
 \frac{1}{\sigma \Delta_n^H} X_n
\end{equation*}
is stationary centered Gaussian and does not depend on $\sigma$ and $\Delta_n$
by the self-similarity property.
Its covariance matrix is  $T_n(H)$ defined in the previous section. 
Denote by $\ell_n(H,\sigma)$ the log-likelihood ratio
\begin{equation*}
 \begin{split}
 \ell_n(H,\sigma) = &
  -\frac{n}{2}\log(2\pi) - nH\log \Delta_n - n\log \sigma \\
  & - \frac{1}{2}\log | T_n(H) | - \frac{1}{2\sigma^2 \Delta_n^{2H}}\langle
  X_n,T_n(H)^{-1} X_n \rangle.
 \end{split}
\end{equation*}
Let
\begin{equation*}
 \begin{split}
 &A_n = \sqrt{n}
  \left\{
\frac{1}{n\sigma^2 \Delta_n^{2H}}\langle
  X_n,T_n(H)^{-1} X_n \rangle -1
		\right\}, \\
  &  B_n =  \frac{1}{\sqrt{n}} \left\{
  \frac{1}{2} \partial_H
  \log | T_n(H) | +
 \frac{1}{2\ \sigma^2 \Delta_n^{2H}}\langle
  X_n,\partial_H\{T_n(H)^{-1}\} X_n \rangle\right\}.
 \end{split}
\end{equation*}
We use the same notation as in the previous section because their distributions are the same.
In particular, 
we have that $(A_n,B_n) \to (A,B)$ in law under $P^n_{(H,\sigma)}$
for a nondegenerate Gaussian random variable $(A,B)$
whose covariance is given by (\ref{AB}). In this setting, the score function is given by
\begin{equation*}
 \nabla \ell_n = 
  \begin{pmatrix}
   \partial_H \ell_n \\
   \partial_\sigma \ell_n 
  \end{pmatrix}
     = \begin{pmatrix}
	A_n \sqrt{n} \log \Delta_n - B_n \sqrt{n}  \\
     A_n \sqrt{n} \sigma^{-1}
       \end{pmatrix}.
\end{equation*}
From this expression, we clearly see that the leading terms of
$\partial_H \ell_n$ and $\partial_\sigma \ell_n$ are linearly dependent, which is exactly the reason why we have a singular Fisher matrix when considering diagonal rate matrices as in \cite{Kawai}. Now, we state the main result of the paper.

 \begin{thm}
 Suppose $\inf_n n\Delta_n > 0$.
Consider a matrix
\begin{equation*}
	      \varphi_n = \varphi_n(H,\sigma) = \begin{pmatrix}
			   \alpha_n & \hat{\alpha}_n \\
	      \beta_n & \hat{\beta}_n
			  \end{pmatrix}
\end{equation*}
with the following properties :
\begin{enumerate}
 \item $|\varphi_n| = \alpha_n \hat{\beta}_n - \hat{\alpha}_n\beta_n \neq 0$.
 \item $\alpha_n \sqrt{n} \to \alpha$ for some $\alpha \geq 0$.
 \item $\hat{\alpha}_n \sqrt{n} \to \hat{\alpha}$ for some $\hat{\alpha} \geq 0$.
 \item $\gamma_n := \alpha_n \sqrt{n} \log \Delta_n + \beta_n\sqrt{n}\sigma^{-1} \to \gamma$ for some $\gamma \in \mathbb{R}$.
 \item $\hat{\gamma}_n := \hat\alpha_n \sqrt{n} \log \Delta_n + \hat\beta_n\sqrt{n}\sigma^{-1} \to \hat{\gamma}$ for some $\hat{\gamma} \in \mathbb{R}$.
 \item $\alpha \hat\gamma - \hat\alpha \gamma \neq 0$.
\end{enumerate}
 Then, the family $\left\{P^n_{(H,\sigma)} ; (H,\sigma) \in (0,1) \times (0,\infty)\right\}$ is LAN at any $(H,\sigma)$ for $\varphi_n(H,\sigma)$ defined previously and
 \begin{equation*}
  I(H,\sigma) =
   \begin{pmatrix}
 \gamma & - \alpha \\ \hat{\gamma} & -\hat{\alpha} 
\end{pmatrix}
 J(H)    \begin{pmatrix}
    \gamma & \hat\gamma  \\   -\alpha &  - \hat\alpha 
   \end{pmatrix}, 
 \end{equation*}
 where $J(H)$ is defined by (\ref{AB}).
 \end{thm}

 \begin{proof}
Let $\theta_0 = (H_0,\sigma_0)$ and $u \in \mathbb{R}^2$.
Let $ \bar{B}(\theta_0,r) = \{\theta \in (0,1)\times (0,\infty) ; |\theta-\theta_0|\leq r\}$ for $r > 0$.
By Taylor's formula,
\begin{equation*}
 \ell_n(\theta_0 + \varphi_n u) -\ell_n(\theta_0)
  = \langle \varphi_n^\ast \nabla \ell_n(\theta_0),  u \rangle +  \frac{1}{2}\langle u, \varphi_n^\ast \nabla^2 \ell_n(\theta_0) \varphi_n u \rangle + r_n,
\end{equation*}
where
\begin{equation*}
 |r_n| \leq \frac{1}{6} 
  |\varphi_n u|\  |u|^2  \max\{ | \varphi_n^\ast\nabla^3 \ell_n(\theta)\varphi_n | ; \theta \in \bar{B}(\theta_0,|\varphi_n u|)\}.
\end{equation*}
Step 1). Let us show that $\varphi_n^\ast \nabla \ell_n(\theta_0) \to \mathcal{N}(0,I(\theta_0))$.
Note that
\begin{equation*}
 \varphi_n^\ast \nabla \ell_n
     = \begin{pmatrix}
A_n \gamma_n - B_n \alpha_n \sqrt{n}
	\\
     A_n \hat\gamma_n - B_n \hat\alpha_n \sqrt{n}
       \end{pmatrix}.
\end{equation*}
Therefore, it converges in law to
\begin{equation*}
\begin{pmatrix}
A \gamma - B \alpha
	\\
 A \hat\gamma - B\hat\alpha
\end{pmatrix} =
\begin{pmatrix}
 \gamma & - \alpha \\ \hat{\gamma} & -\hat{\alpha} 
\end{pmatrix}
  \begin{pmatrix}
   A \\ B
  \end{pmatrix}
,
\end{equation*}
which is Gaussian.

\noindent
Step 2). Here we compute
\begin{equation*}
 \varphi_n^* \nabla^2 \ell_n \varphi_n.
\end{equation*}
Let
\begin{equation*}
 \begin{split}
& C_n = \frac{1}{n\sigma^2 \Delta_n^{2H}}\langle
  X_n, T_n(H)^{-1}X_n \rangle, \\
&  D_n = \frac{1}{n\sigma^2 \Delta_n^{2H}}\langle
  X_n, \partial_H \{T_n(H)^{-1}\}X_n \rangle, \\
  &
 E_n =  \frac{1}{n} \left\{
  \frac{1}{2} \partial_H^2
  \log | T_n(H) | +
 \frac{1}{2\ \sigma^2 \Delta_n^{2H}}\langle
  X_n,\partial_H^2\{T_n(H)^{-1}\} X_n \rangle\right\}.  
 \end{split}
\end{equation*}
Note that
\begin{equation*}
 \partial_H
  \log | T_n(H) | = - \mathrm{tr}(\partial_H\{T_n(H)^{-1}\}T_n(H))
\end{equation*}
and so,
\begin{equation*}
\partial_H^2
 \log | T_n(H) | =  -\mathrm{tr}(\partial_H^2T_n(H)^{-1}T_n(H))
 + \mathrm{tr}(T_n(H)^{-1}\partial_HT_n(H) T_n(H)^{-1}\partial_HT_n(H)). 
\end{equation*}
Then, the $(1,1)$ element of $\nabla^2 \ell_n \varphi_n$ is
\begin{equation*}
 \begin{split}
& \alpha_n
  \left(
\sqrt{n} \log \Delta_n \partial_HA_n - \sqrt{n}\partial_HB_n
	 \right)
  + \beta_n \frac{\sqrt{n}}{\sigma}\partial_HA_n   \\
  & = \gamma_n \partial_HA_n - \alpha_n\sqrt{n} \partial_HB_n \\
  &= \gamma_n \sqrt{n}\left( -2 C_n \log \Delta_n  + D_n\right)
  - \alpha_n n \left(
  E_n - D_n \log \Delta_n
  \right)
 \end{split}
\end{equation*}
and the $(1,2)$ element is
\begin{equation*}
 \begin{split}
& \alpha_n
  \frac{\sqrt{n}}{\sigma}\partial_HA_n + \beta_n \left(
-\frac{\sqrt{n}}{\sigma^2} + \frac{\sqrt{n}}{\sigma}\partial_\sigma A_n
  \right) \\
  &=
  -2 \frac{\sqrt{n}}{\sigma}C_n \gamma_n
  + \frac{n}{\sigma} D_n\alpha_n - \beta_n\frac{\sqrt{n}}{\sigma^2}.
 \end{split}
\end{equation*}
It follows then that the $(1,1)$ element of
$ \varphi_n^\ast \nabla^2 \ell_n \varphi_n $ is
\begin{equation*}
 -2C_n \gamma_n^2 + 2D_n \gamma_n \alpha_n \sqrt{n}
  -\alpha_n^2nE_n - \beta_n^2\frac{\sqrt{n}}{\sigma^2}.
\end{equation*}
From this, it is clear that
the $(2,2)$ element of
$ \varphi_n^\ast \nabla^2 \ell_n \varphi_n$ is
\begin{equation*}
 -2C_n \hat\gamma_n^2 + 2D_n \hat\gamma_n \hat\alpha_n \sqrt{n}
  -\hat\alpha_n^2nE_n - \hat\beta_n^2\frac{\sqrt{n}}{\sigma^2}.
\end{equation*}
Also it is already not difficult to see that
the $(1,2)$ element of
$ \varphi_n^\ast \nabla^2 \ell_n \varphi_n$ is
\begin{equation*}
 -2C_n \gamma_n \hat\gamma_n + D_n( \hat\gamma_n \alpha_n \sqrt{n} +
   \gamma_n \hat\alpha_n \sqrt{n})
  -\alpha_n\hat\alpha_nnE_n - \beta_n \hat\beta_n\frac{\sqrt{n}}{\sigma^2}.
\end{equation*}
It is clear that $nC_n \sim \chi^2(n)$ and so, $C_n \to 1$.
By the same argument as the proof of Lemma~3.5 of \cite{Cohen}, we
have
\begin{equation*}
 D_n \to - \frac{1}{2\pi} \int_{-\pi}^{\pi} \partial_H \log f_H(\lambda)\mathrm{d}\lambda = E[AB]
\end{equation*}
and
\begin{equation*}
 E_n \to  \frac{1}{4\pi} \int_{-\pi}^{\pi} |\partial_H \log f_H(\lambda)|^2 \mathrm{d}\lambda = E[B^2].
\end{equation*}
Therefore, 
\begin{equation*}
 \begin{split}
& \varphi_n^\ast \nabla^2 \ell_n \varphi_n \\
&       \to - \begin{pmatrix}
2\gamma^2 - 2\gamma \alpha E[AB] + \alpha^2 E[B^2]
	    &
	      2\gamma\hat{\gamma} - (\gamma \hat\alpha + \hat\gamma \alpha) E[AB] + \alpha\hat{\alpha} E[B^2]  \\
	      2\gamma\hat{\gamma} - (\gamma \hat\alpha + \hat\gamma \alpha) E[AB] + \alpha\hat{\alpha} E[B^2] &
	      2\hat\gamma^2 - 2\hat\gamma \hat\alpha E[AB] + \hat\alpha^2 E[B^2] 
	   \end{pmatrix}. 
 \end{split}
\end{equation*}
This coincides with $ - I(H,\sigma)$ because $E[A^2] = 2$.
\\
Step 3) It remains to show $r_n \to 0$.
From the computation in Step 2, we deduce that 
the tensor $-\varphi_n^\ast  \nabla^3 \ell_n \varphi_n$ consists of the vectors
\begin{equation*}
 \begin{split}
  & 2 \nabla C_n \gamma_n^2 - 2\gamma_n \alpha_n\sqrt{n} \nabla D_n + \alpha_n n
   \nabla E_n,
  \\
  & 2\nabla C_n \gamma_n\hat{\gamma}_n - (\gamma_n \hat\alpha_n\sqrt{n} + \hat\gamma_n \alpha_n\sqrt{n}) \nabla D_n + \alpha_n\hat{\alpha}_nn \nabla E_n,  \\
  & 2\nabla C_n \hat\gamma_n^2 - 2\hat\gamma_n \hat\alpha_n\sqrt{n} \nabla D_n + \hat\alpha_n^2n\nabla E_n.
 \end{split}
\end{equation*}
By the same argument as the proof of Lemma~3.7 of \cite{Cohen},
we have that there exists $\epsilon > 0$ such that
$\nabla C_n, \nabla D_n $ and $\nabla E_n$ are 
 of $O_p(n^{1/2-\epsilon}|\log \Delta_n|)$ uniformly in $\theta \in\bar{B}(\theta_0,\epsilon)$.
 On the other hand,
\begin{equation*}
 \begin{split}
  &     n(\alpha_n^2 + \hat\alpha_n^2)\to \alpha^2 + \hat\alpha^2, \\
  & \frac{n}{|\log \Delta_n| }(\beta_n^2 + \hat{\beta}_n^2)\to \sigma^2(\alpha^2 + \hat\alpha^2),
 \end{split}
\end{equation*} which implies that
 $|\varphi_n u| = O(\sqrt{|\log |\Delta_n|/n})$.
 Since $\inf_n n \Delta_n > 0$ by the assumption,  we conclude that $r_n \to 0$.

\end{proof}

Several examples can be elicited. A symmetric matrix for rate $\varphi_n $ is
\begin{equation*}
	      \varphi_n = \begin{pmatrix}
\frac{1}{\sqrt{n}\log \Delta_n} & \frac{1}{\sqrt{n}}
			   \\
			   \frac{1}{\sqrt{n}} & - \frac{\sigma \log \Delta_n}{\sqrt{n}}
	      
			  \end{pmatrix},
\end{equation*}
for which $\alpha = 0$, $\hat{\alpha} = 1$, $\gamma = 1 + \sigma^{-1}$
and $\hat{\gamma} = 0$.
Another example is that
\begin{equation*}
	      \varphi_n = \begin{pmatrix}
			   \frac{1}{\sqrt{n}} & 
			   \frac{1}{\sqrt{n}} \\
			   \frac{\sigma}{\sqrt{n}}
			   \left(\gamma -\log \Delta_n\right) & \frac{\sigma}{\sqrt{n}}
			   \left(\hat{\gamma} -\log \Delta_n\right)
			  \end{pmatrix}
\end{equation*}
for $(\gamma,\hat\gamma)$  with $\gamma \neq \hat\gamma$, for which
$\alpha = \hat{\alpha} = 1$.

Two other examples of rate matrix will be used in the following section, namely
\begin{equation*}
	      \varphi_n = \frac{1}{\sqrt{n}}\begin{pmatrix}

					    1			  &
					    0 \\ - \sigma \log \Delta_n & 1
	      
			  \end{pmatrix},
\end{equation*}
which gives $\alpha = 1$, $\hat{\alpha}=0$, $\gamma=0$ and
 $\hat{\gamma} = \sigma^{-1}$ and
 \begin{equation*}
	      \varphi_n = \frac{1}{\sqrt{n}}\begin{pmatrix}

					   \frac{1}{\log \Delta_n}		  & 
					    1 \\ 0  &  -\sigma \log \Delta_n
	      
			  \end{pmatrix},
\end{equation*}
which gives $\alpha = 0$, $\hat{\alpha}=1$, $\gamma=1$ and
 $\hat{\gamma} = 0$.  Remark that all the examples are non-diagonal rate matrix depending on the parameter $\sigma$.

\section{Efficient rates of estimation and optimality of ML estimtors}

\subsection{The efficient estimation rate for $H$}
As the rate matrix, we can take
\begin{equation*}
	      \varphi_n = \frac{1}{\sqrt{n}}\begin{pmatrix}

					    1			  &
					    0 \\ - \sigma \log \Delta_n & 1
	      
			  \end{pmatrix},
\end{equation*}
which gives $\alpha = 1$, $\hat{\alpha}=0$, $\gamma=0$ and
 $\hat{\gamma} = \sigma^{-1}$.  It is worth mentioning that the rate matrix is not diagonal and depends on the parameter $\sigma$.
By Theorem~\ref{thm:minimax}, the LAN property implies that
 \begin{equation*}
\liminf_{n\to \infty}  E_{(H,\sigma)}\left[l\left(\varphi_n^{-1}
\begin{pmatrix}
 \hat{H}_n-H \\ \hat{\sigma}_n-\sigma
\end{pmatrix}
			     \right)\right] \geq
E\left[
  l \left(
\begin{pmatrix}
 E[B^2] & - E[AB]/\sigma \\
 - E[AB]/\sigma   & 2/\sigma^2
\end{pmatrix}^{-1/2}  N
    \right)
 \right]
 \end{equation*}
for any loss function $l$ satisfying the condition given in Theorem~\ref{thm:minimax},
 where $N \sim \mathcal{N}(0,I_2)$.
 Since
 \begin{equation*}
  \varphi_n^{ -1}
	      = \sqrt{n}\begin{pmatrix}
			 1 &  0 \\
			 \sigma \log \Delta_n & 1
			\end{pmatrix},
 \end{equation*}
 we obtain the asymptotic lower bound of
 \begin{equation*}
n  E_{(H,\sigma)}[(\hat{H}_n-H)^2]
 \end{equation*}
by taking $l(x,y) = x^2$.
 This means that the efficient rate of estimation for $H$ is $\sqrt{n}$ when both $H$ and $\sigma$ are unknown. 
Note that when $\sigma$ is known, the efficient rate for $H$ is $\sqrt{n}\log \Delta_n$, which follows from, say, \cite{Kawai}.

 \subsection{The efficient estimation rate for $\sigma$}
 As the rate matrix, we can take
\begin{equation*}
	      \varphi_n = \frac{1}{\sqrt{n}}\begin{pmatrix}

					   \frac{1}{\log \Delta_n}		  & 
					    1 \\ 0  &  -\sigma \log \Delta_n
	      
			  \end{pmatrix},
\end{equation*}
which gives $\alpha = 0$, $\hat{\alpha}=1$, $\gamma=1$ and
 $\hat{\gamma} = 0$.   It is worth repeating that the rate matrix is not diagonal and depends on the parameter $\sigma$.
Again by Theorem~\ref{thm:minimax}, the LAN property implies that
 \begin{equation*}
\liminf_{n\to \infty}  E_{(H,\sigma)}\left[l\left(\varphi_n^{-1}
\begin{pmatrix}
 \hat{H}_n-H \\ \hat{\sigma}_n-\sigma
\end{pmatrix}
			     \right)\right] \geq
E\left[
  l \left(
\begin{pmatrix}
 2 & - E[AB] \\
 - E[AB]  & E[B^2]
\end{pmatrix}^{-1/2}N
    \right)
 \right]
 \end{equation*}
for any loss function $l$ satisfying the condition given in Theorem~\ref{thm:minimax},
 where $N \sim \mathcal{N}(0,I_2)$.
 Since
 \begin{equation*}
  \varphi_n^{ -1}
	      = -\frac{\sqrt{n}}{\sigma}\begin{pmatrix}
					 -\sigma \log \Delta_n &  -1\\
			 0 & \frac{1}{\log \Delta_n}
			\end{pmatrix},
 \end{equation*}
 we obtain the asymptotic lower bound of
 \begin{equation*}
\frac{n}{(\log \Delta_n)^2}  E_{(H,\sigma)}[(\hat{\sigma}_n-\sigma)^2]
 \end{equation*}
by taking $l(x,y) = y^2$.
 This means that the efficient rate of estimation for $\sigma$ is $\sqrt{n}/|\log \Delta_n|$ when both $H$ and $\sigma$ are unknown. 
Note that when $H$ is known, the efficient rate for $\sigma$ is $\sqrt{n}$, which follows from, say, \cite{Kawai}.

\end{document}